\documentclass[a4paper, 11pt]{amsart}

\let\oldtocsection=\tocsection

\let\oldtocsubsection=\tocsubsection

\let\oldtocsubsubsection=\tocsubsubsection

\renewcommand{\tocsection}[2]{\hspace{0em}\oldtocsection{#1}{#2}}
\renewcommand{\tocsubsection}[2]{\hspace{2em}\oldtocsubsection{#1}{#2}}
\renewcommand{\tocsubsubsection}[2]{\hspace{2em}\oldtocsubsubsection{#1}{#2}}

\newcommand{\Cn}{\textup{\textrm{\v{c}}}}

\title{{On the squares functor and the Gaitsgory--Rozenblyum conjectures}}
\author{F\'elix Loubaton and Jaco Ruit}
\address{Max-Planck-Institut f\"ur Mathematik, Vivatsgasse 7, Bonn, Germany}
\email{loubaton@mpim-bonn.mpg.de}
\address{Max-Planck-Institut f\"ur Mathematik, Vivatsgasse 7, Bonn, Germany}
\email{ruit@mpim-bonn.mpg.de} 

\usepackage{multicol}
\usepackage{amssymb}
\usepackage{amsfonts}
\usepackage{amsthm}
\usepackage{float}
\usepackage{kpfonts} 
\usepackage[mathscr]{eucal}

\usepackage[english]{babel}

\usepackage{tikz-cd}{}  
\usetikzlibrary{backgrounds}
\usetikzlibrary{patterns}
\usepackage{enumitem}
\usepackage[linktocpage]{hyperref}
\usepackage{aliascnt}
\usepackage{varioref}
\usepackage{xstring}
\usepackage{xcolor}
\usepackage{bbm}
\usepackage{breakurl}
\usepackage{todonotes}
\usepackage{mathdots}
\usepackage{float}

\usepackage[left=3cm,right=3cm,top=3.2cm,bottom=3.2cm]{geometry} 

\newtheorem{thmintro}{Theorem}

\newcommand{\inewtheorem}[2]{
	\newaliascnt{#1}{thmintro}
	\newtheorem{#1}[#1]{#2}
	\aliascntresetthe{#1}
}
\inewtheorem{propintro}{Proposition}
\inewtheorem{corintro}{Corollary}

\newtheorem{theorem}{Theorem}
\newcommand{\jnewtheorem}[2]{
	\newaliascnt{#1}{theorem}
	\newtheorem{#1}[#1]{#2}
	\aliascntresetthe{#1}
}
\numberwithin{theorem}{section}
\jnewtheorem{lemma}{Lemma}
\jnewtheorem{proposition}{Proposition}
\jnewtheorem{corollary}{Corollary}

\theoremstyle{definition}
\jnewtheorem{definition}{Definition}
\jnewtheorem{example}{Example}
\jnewtheorem{remark}{Remark}
\jnewtheorem{construction}{Construction}
\jnewtheorem{notation}{Notation}
\jnewtheorem{convention}{Convention}

\labelformat{theorem}{Theorem #1}
\labelformat{lemma}{Lemma #1}
\labelformat{proposition}{Proposition #1}
\labelformat{corollary}{Corollary #1}
\labelformat{definition}{Definition #1}
\labelformat{example}{Example #1}
\labelformat{section}{Section #1}
\labelformat{subsection}{Subsection #1}
\labelformat{equation}{(#1)}
\labelformat{remark}{Remark #1}
\labelformat{chapter}{Chapter #1}
\labelformat{construction}{Construction #1}  
\labelformat{conjecture}{Conjecture #1}
\labelformat{notation}{Notation #1}
\labelformat{convention}{Convention #1}
\labelformat{thmintro}{Theorem #1}
\labelformat{corintro}{Corollary #1}
\labelformat{propintro}{Corollary #1}

\numberwithin{table}{subsection}
\labelformat{table}{Table #1}

\hypersetup{
	pdftitle={On the squares functor and the Gaitsgory--Rozenblyum conjectures},
	pdfauthor={F\'elix Loubaton and Jaco Ruit},
	bookmarksnumbered=true,     
	bookmarksopen=true,         
	bookmarksopenlevel=1,       
	colorlinks=true,   
	linkcolor = {cyan!60!black},
	urlcolor = {cyan!60!black},
	citecolor = {cyan!60!black},   
	pdfencoding=unicode,      
	pdfstartview=Fit,           
	pdfpagemode=UseOutlines,    
	pdfpagelayout=OneColumn,
	breaklinks
}

\newcommand{\op}{\mathrm{op}}
\newcommand{\co}{\textup{\textrm{2-op}}}
\newcommand{\map}{\mathrm{Map}}
\def\colim{\qopname\relax m{colim}}

\newcommand{\fun}{\mathrm{Fun}}

\newcommand{\Cat}{\mathrm{Cat}}

\renewcommand{\S}{\mathscr{S}}
\renewcommand{\D}{\mathscr{D}}
\renewcommand{\P}{\mathscr{P}}
\newcommand{\Q}{\mathscr{Q}}
\newcommand{\E}{\mathscr{E}}

\newcommand{\Gr}{\mathrm{Gr}}

\newcommand{\tp}{\mathrm{t}}

\newcommand{\DbliCat}{\mathrm{Dbl}\infty\Cat}
\newcommand{\CDbliCat}{\mathrm{CDbl}\infty\Cat}

\newcommand{\PSh}{\mathrm{PSh}}
\newcommand{\Sq}{\mathrm{Sq}}	
\newcommand{\id}{\mathrm{id}}
\newcommand{\vop}{\mathrm{vop}}

\newcommand{\hop}{\mathrm{hop}}
\newcommand{\vrectangle}{{\ooalign{\lower.3ex\hbox{$\sqcup$}\cr\raise.4ex\hbox{$\sqcap$}}}}

\newcommand{\Str}{\mathrm{Str}}

\newcommand{\Filt}{\mathrm{Filt}}

\newcommand{\C}{\mathscr{C}}
\newcommand{\uvar}{ \_ }

\newcommand{\Gaunt}{\mathrm{Gaunt}}

\begin{document}

\begin{abstract}
In the seminal work of Gaitsgory and Rozenblyum on derived algebraic geometry, eight conjectures regarding the theory of $(\infty,2)$-categories are stated. This paper aims to clarify the status of these claims, and to provide a proof for the last remaining open one. Along the way, we demonstrate the universal property of the so-called squares functor, a construction that plays an important role in the $(\infty,2)$-categorical foundations of Gaitsgory--Rozenblyum.
\end{abstract}

\maketitle

\setcounter{tocdepth}{1}

\tableofcontents

\section{Introduction}

In the appendix of the celebrated work of Gaitsgory and Rozenblyum on derived algebraic geometry \cite{GR}, the necessary foundations of $(\infty,2)$-categories are developed. These specifically concern the Gray tensor product, a fundamental operation in the theory of $(\infty,2)$-categories, and the squares functor that provides a non-trivial way to create a double $\infty$-category from a $(\infty,2)$-category by using the Gray tensor product. There are eight statements in the $(\infty,2)$-categorical foundations of \cite{GR} that are conjectured by Gaitsgory and Rozenblyum. Since then, mathematicians have made efforts to provide proofs of these conjectures. The goal of this paper is to clarify the status of the conjectures, as well as to provide a demonstration of the last unproven conjecture, so that we may conclude that all conjectures are now resolved. We will also provide a complete picture of the squares construction by demonstrating its universal property, generalizing a similar result for strict double categories by Grandis and Par\'e \cite{GrandisPare}.

We will continue this introduction by providing an informal overview of the results that we will cover in this paper. The precise statements and definitions can be found in the text.

\subsection{$(\infty,2)$-Categories and double $\infty$-categories.} Throughout this paper, we will study and use $(\infty,2)$-categories and double $\infty$-categories. These are two different generalizations of $\infty$-categories  that both add a notion of non-invertible two-dimensional cells. Imprecisely, an $(\infty,2)$-category $\C$ consists of the following data:
\begin{itemize}
    \item a space of objects,
    \item a space of arrows or $1$-cells $$\begin{tikzcd}
	a & b,
	\arrow["f", from=1-1, to=1-2]
\end{tikzcd}$$
    \item a space of $2$-cells
$$
\begin{tikzcd}
	a & b.
	\arrow[""{name=0, anchor=center, inner sep=0}, "f", bend left=40, from=1-1, to=1-2]
	\arrow[""{name=1, anchor=center, inner sep=0}, "g"', bend right=40, from=1-1, to=1-2]
	\arrow[Rightarrow, from=0, to=1, shorten <=3pt, shorten >=3pt]
\end{tikzcd}$$
\end{itemize}
In contrast, a double $\infty$-category $\P$ has two directions of 1-cells, and consists of the following data:
\begin{itemize}
    \item a space of objects,
    \item a space of horizontal arrows or 1-cells $$\begin{tikzcd}
	a & b,
	\arrow["F", from=1-1, to=1-2]
\end{tikzcd}$$
    \item a space of vertical arrows or 1-cells
$$\begin{tikzcd}
	a \\
	b,
	\arrow["f"', from=1-1, to=2-1]
\end{tikzcd}$$
    \item a space of $2$-cells
$$\begin{tikzcd}
        a \arrow[r,"F"name=f]\arrow[d, "f"'] & b \arrow[d, "g"] \\
        c \arrow[r,"G"'name=t] & d.
        \arrow[from=f,to=t, Rightarrow, shorten <= 6pt, shorten >= 6pt]
    \end{tikzcd}
$$

\end{itemize}
Moreover, the cells of $(\infty,2)$-categories and double $\infty$-categories have various coherent composition laws, and satisfy certain \textit{completeness} or \textit{univalence} conditions. We will recall the precise definitions of these structures in \ref{section:prelims}.

\subsection{The Gray tensor product for $(\infty,2)$-categories.}
\label{subsection:introduction on Gray}
The Gray tensor product is  a fundamental operation allowing to define, control, and study \textit{lax}
phenomena occurring naturally in the study of $(\infty,2)$-categories. Informally, given a concept in category theory (e.g.\ natural transformations or colimits), the lax
variant is obtained by replacing all commutative diagrams appearing in the definition of this concept by
diagrams that commute up to cells that are a priori \textit{not}  invertible. In this sense, the Gray tensor product may be viewed as
the lax variant of the cartesian product. It is usually denoted by the symbol $\otimes$.
For example, the cartesian product $[1] \times [1]$ corresponds to the free-living commutative square
\[\begin{tikzcd}
	00 & 10 \\
	01 & 11,
	\arrow[from=1-1, to=1-2]
	\arrow[from=1-1, to=2-1]
	\arrow[from=1-2, to=2-2]
	\arrow["\circlearrowleft"{description}, draw=none, from=2-1, to=1-2]
	\arrow[from=2-1, to=2-2]
\end{tikzcd}\]
while the Gray tensor product $[1]\otimes [1]$ is the free-living \textit{lax commutative square}
\[\begin{tikzcd}
	00 & 10 \\
	01 & 11.
	\arrow[from=1-1, to=1-2]
	\arrow[from=1-1, to=2-1]
	\arrow[from=1-2, to=2-2]
	\arrow[ Rightarrow, from=1-2, to=2-1]
	\arrow[from=2-1, to=2-2]
\end{tikzcd}\]

The construction and combinatorics of the Gray tensor product are non-trivial, and there are many ways to define this tensor product. We provide a non-exhaustive list of constructions:

\begin{enumerate}
    \item[$\otimes_V$] The Gray tensor product on the model of \textit{2-complicial sets} of $(\infty,2)$-categories defined by Verity in \cite{verity2008weak}.
    \item[$\otimes_{GR}$] The Gray tensor product on $(\infty,2)$-categories defined by Gaitsgory--Rozenblyum in \cite{GR}.
    \item[$\otimes_{GHL}$] The Gray tensor product on the model of \textit{scaled simplicial sets} of $(\infty,2)$-categories defined by Gagna--Harpaz--Lanari in \cite{gagna2021gray}.
    \item[$\otimes$] The Gray tensor product on the model of \textit{2-quasi-categories} for $(\infty,2)$-categories, that arise as certain set-valued presheaves on $\Theta_2$ (see \ref{def:glob sums}), defined by Maehara in \cite{maehara2021gray}. This will be the one used in this paper, but interpreted directly for presheaves on $\Theta_2$ valued in \textit{spaces}.
    \item[$\otimes_{DKM}$] The Gray tensor product on \textit{$2$-comical sets} defined by Doherty-- Kapulkin--Maehara in \cite{doherty2023equivalence}.
    \item[$\otimes_{CM}$] The Gray tensor product on $(\infty,2)$-categories defined by Campion--Maehara in \cite{campionmaehara}.
    \item[$\otimes_{C}$] The Gray tensor product on $(\infty,2)$-categories defined by Campion in \cite{campion2023gray}.
    \item[$\otimes_L$] The Gray tensor product on $(\infty,\omega)$-categories $\otimes_L^{\omega}$ defined by the first author in \cite{Effectivity} induces a Gray tensor product on $(\infty,2)$-categories after \textit{intelligent $2$-truncation} (see \cite[Definition 1.1.5]{Effectivity}).
\end{enumerate}

The multiplicity of these definitions arises both from the fact that there are many different models of $(\infty,2)$-categories, most of which admit a description of a Gray tensor product, and from the fact that within a chosen model, there are many a priori very different descriptions of this operation.
Fortunately, all these definitions are compared, and thus correspond to the same operation at the level of the $(\infty,1)$-category of $(\infty,2)$-categories:

\[\begin{tikzcd}
	{\otimes_{CM}} & {\otimes_C} & {\otimes_{DKM}} & {\otimes_{GR}} \\
	{\otimes} & {\otimes_L} & {\otimes_V} & {\otimes_{GHL}}
	\arrow["{(3)}"', equals, from=1-2, to=2-2]
	\arrow["{(7)}", equals, from=1-4, to=2-4]
	\arrow["{(2)}"', equals, from=2-1, to=2-2]
	\arrow["{(5)}"', equals, from=2-3, to=1-3]
	\arrow["{(4)}", equals, from=2-3, to=2-2]
	\arrow["{(6)}"', equals, from=2-3, to=2-4]
	\arrow["{(1)}"', equals, from=1-1, to=2-1]
\end{tikzcd}\]
\begin{multicols}{2}
    \begin{enumerate}
        \item \cite[Remark 3.4]{campionmaehara}
        \item \ref{prop:comparaison of gray tensor product}
        \item \cite[Remark 1.4.19]{Effectivity}
        \item \cite[Remark 1.4.19]{Effectivity} 
        \item \cite[Theorem 6.5]{campion2020comical}
        \item \cite[Corollary 2.11]{gagna2021gray}
        \item \cite[Theorem 6.26]{abellan2023comparing}
    \end{enumerate}
\end{multicols}

\noindent As we will show in \ref{prop:auto of gray}, the Gray tensor product admits no non-trivial automorphisms, and so all of the above identifications are necessarily unique.

In \cite{GR}, the following statements are conjectured about the Gray tensor product. Taking advantage of the various descriptions of this operation, these are all proved.
\begin{table}[H]
	\centering
	\renewcommand{\arraystretch}{1.2}
	\makebox[\textwidth]{%
	\begin{tabular}{| c | p{6cm} | p{5cm} |}
		\hline
		\textbf{Conjecture} & \textbf{Description} & \textbf{Status} \\
		\hline
		Proposition 10.3.2.9 & The Gray tensor product is associative. &  This was shown by Verity in \cite[Lemma 131]{verity2008weak}.  \\	\hline 
		Proposition 10.3.2.6 & The Gray tensor product commutes with colimits in both variables. & This was shown by Ozornova, Rovelli and Verity in \cite[Corollary 2.6]{ozornova2020gray}, building on the work of Verity \cite{verity2008weak}. \\
		\hline 
		Proposition 10.3.3.5 & The iterated Gray tensor product of simplices is a (strict) 2-category. &  This was shown by Maehara in \cite[Corollary 7.11]{maehara2021gray}.  \\
		\hline
	\end{tabular}}
	\caption{The Gaitsgory--Rozenblyum conjectures, part 1.}\label{table:gray conjectures}
\end{table}
\noindent It should be noted that the first two conjectures are proven for almost all of the different definitions of the Gray tensor product by the authors who introduced them. We have chosen to focus on Verity's Gray tensor product here because it is, to our knowledge, the first definition of a Gray tensor product in a homotopical setting.

\subsection{The squares functor.} The \textit{squares construction} was originally introduced by Ehresmann \cite{ehresmann} for strict double categories. Its $\infty$-categorical incarnation plays an important role in the $(\infty,2)$-categorical set-up of Gaitsgory--Rozenblyum, where it was defined in \cite[Subsection 10.4.1]{GR}. If $\C$ is an $(\infty,2)$-category, then the double $\infty$-category $\Sq(\C)$ of squares in $\C$ is loosely described as follows:
\begin{itemize}
    \item its objects are those of $\C$,
    \item its horizontal and vertical arrows are given by the arrows of $\C$,
    \item its 2-cells correspond to lax commutative squares in $\C$, so that there is a one-to-one correspondence as pictured below:
    \[ 
    \begin{tikzcd}
        a \arrow[r,"F"name=f]\arrow[d, "f"'] & b \arrow[d, "g"] \\
        c \arrow[r,"G"'name=t] & d
        \arrow[from=f,to=t, Rightarrow, shorten <= 6pt, shorten >= 6pt]
    \end{tikzcd}
    ~~ \text{in $\Sq(\C)$} \quad \Leftrightarrow  \quad \begin{tikzcd}
        a \arrow[r,"F"]\arrow[d, "f"'] & |[alias=f]|b \arrow[d, "g"] \\
        |[alias=t]|c \arrow[r,"G"'] & d
        \arrow[from=f,to=t, Rightarrow, shorten <= 6pt, shorten >= 6pt]
    \end{tikzcd}
    ~~ \text{in $\C$}.
    \]
\end{itemize}
This is functorial in $\C$, and we will provide a precise definition in \ref{section:cech and sq}. In fact, we will first introduce a {relative} version of the above construction that was considered by the first author in \cite{Effectivity}, and can be viewed as a 2-categorification of the \v{C}ech nerve. This is a slightly more general version than the relative squares construction $\Sq^\mathrm{Pair}$ that was earlier defined in \cite[Subsection 10.4.3]{GR} for pairs of $(\infty,2)$-categories.

In \cite{GR}, the following conjectures about the squares functor appear:
\begin{table}[H]
	\centering
	
	\renewcommand{\arraystretch}{1.2}
	\makebox[\textwidth]{%
	\begin{tabular}{| c | p{6cm} | p{5cm} |}
		\hline
		\textbf{Conjecture} & \textbf{Description} & \textbf{Status} \\
		\hline 
		Theorem 10.4.1.3 & The functor $\Sq$ is fully faithful. & This was shown by Abell\'an in \cite[Theorem 5]{abellan2023comparing}. \\
		\hline
		Theorem 10.4.3.5 & The functor $\Sq^\mathrm{Pair}$ is fully faithful. &  This was shown by the first author in \cite[Proposition 3.4.22]{Effectivity}.   \\
		\hline 
		Theorem 10.5.2.3 & The identification of the essential image of $\Sq$ and $\Sq^\mathrm{Pair}$. & This was shown by the first author in \cite[Proposition 3.4.23]{Effectivity}. \\
		\hline
		Theorem 10.4.6.3 & The \textit{cubes functor} that is closely related to $\Sq$, is fully faithful. & This was shown by the first author in \cite[Proposition 3.4.22]{Effectivity}.  \\
		\hline
	\end{tabular}}
	\caption{The Gaitsgory--Rozenblyum conjectures, part 2.}
\end{table}

\subsection{The universal property of squares} We will further complete the picture of the squares construction by demonstrating its universal property.
Note that one may view the squares construction as a non-trivial way to produce a double $\infty$-category out of an $(\infty,2)$-category $\C$. Besides, there are two trivial ways in which one may view $\C$ as a double $\infty$-category. Namely, one can consider its \textit{vertical inclusion}, denoted by $\C_v$, that is loosely described as follows:
\begin{itemize}
    \item its objects are those of $\C$,
    \item its vertical arrows are given by the arrows of $\C$,
    \item its horizontal arrows are identities,
    \item its $2$-cells correspond to those of $\C$, so that there is a one-to-one correspondence as pictured below: \[ 
    \begin{tikzcd}
        a \arrow[r, equal, ""name=f]\arrow[d, "f"'] & a \arrow[d, "g"] \\
        b \arrow[r,equal, ""'name=t] & b
        \arrow[from=f,to=t, Rightarrow, shorten <= 6pt, shorten >= 6pt]
    \end{tikzcd}
    ~~ \text{in $\C_v$} \quad \Leftrightarrow  \quad \begin{tikzcd}
	a & b
	\arrow[""{name=0, anchor=center, inner sep=0}, "g", bend left=40, from=1-1, to=1-2]
	\arrow[""{name=1, anchor=center, inner sep=0}, "f"', bend right=40, from=1-1, to=1-2]
	\arrow[Rightarrow, from=0, to=1, shorten <=3pt, shorten >=3pt]
\end{tikzcd}
    ~~ \text{in $\C$}.
    \]
\end{itemize}
Dually, the \textit{horizontal inclusion} of $\C$, denoted $\C_h$, is defined similarly by interchanging the horizontal and vertical arrows. We will discuss the details in \ref{subsection:dbl cats}. The ways in which we may view $\C$ as a double $\infty$-category are related by a span of functors
$$\C_v\to \Sq(\C)\leftarrow \C_h$$
between double $\infty$-categories; see \ref{section:cech and sq}.

The first contribution of this paper will be to demonstrate the universal property of this span, which was conjectured in \cite[Section 2.7]{JacoThesis} before. It states that the double $\infty$-category $\Sq(\C)$ is obtained by freely adding so-called \textit{companions} to $\C_v$. The notion of companions will be recalled in \ref{def:comp}. They were first introduced in the context of strict double categories by Grandis and Par\'e \cite{GrandisPare}, and later considered by Gaitsgory and Rozenblyum \cite{GR} for double $\infty$-categories to describe the essential image of $\Sq$.  Precisely, we will show the following in \ref{subsection:uni prop}:

\begin{thmintro}\label{thmintro:uni prop}
    Let $\C$ be an $(\infty,2)$-category, and $\Q$ be a double $\infty$-category. Then the canonical map $\C_v \to \Sq(\C)$ induces a monomorphism 
     $$
    \map_{\DbliCat}(\Sq(\C), \Q) \to \map_{\DbliCat}(\C_v, \Q)
    $$
    whose image is given by the functors $\C_v \to \Q$ that carry every arrow in $\C$ to a vertical arrow of $\Q$ that admits a companion horizontal arrow.
    
    Dually and under the assumption that $\Q$ is a locally complete (see \ref{subsection:dbl cats}), the canonical map $\C_h \to \Sq(\C)$ induces a monomorphism 
     $$
    \map_{\DbliCat}(\Sq(\C), \Q) \to \map_{\DbliCat}(\C_h, \Q)
    $$
    whose image is given by the functors $\C_h \to \Q$ that carry every arrow in $\C$ to a horizontal arrow of $\Q$ that is the companion of a vertical arrow.
\end{thmintro}

The above result generalizes \cite[Theorem B]{CompJaco}, where the above was shown to hold for $\C = [1]$. To prove \ref{thmintro:uni prop},  we show a slightly more general result for the {directed \v{C}ech nerve} that was introduced in \cite{Effectivity}; see \ref{thm:uni prop of cech}. In the context of strict double cateogries, \ref{thmintro:uni prop} was shown by Grandis--Par\'e \cite[Theorem 1.8]{GrandisPare}.

\subsection{The remaining Gaitsgory--Rozenblyum conjecture.} We may complete the tables of the Gaitsgory--Rozenblyum conjectures with the following:

\begin{table}[H]
	\centering
	\renewcommand{\arraystretch}{1.2}
	\makebox[\textwidth]{%
	\begin{tabular}{| c | p{6cm} | p{5cm} |}
		\hline
		\textbf{Conjecture} & \textbf{Description} & \textbf{Status} \\
		\hline 
		Proposition 10.4.5.4 & An equation relating the  Gray tensor product and the squares construction. & This will be shown in this paper, appearing as \ref{thmintro:comparison}.\\
		\hline
	\end{tabular}}
	\caption{The Gaitsgory--Rozenblyum conjectures, part 3.}\label{table:gray conjectures}
\end{table}

\noindent As of yet,  no proof of \cite[Proposition 10.4.5.4]{GR}  has appeared in the literature. Its demonstration will be our second and last contribution of this paper:

\begin{thmintro}
    \label{thmintro:comparison}
    Let $\C,\D$ and $\E$ be $(\infty,2)$-categories. There exists a natural equivalence 
    $$
    \map(\C_h\times \D_v,\Sq(\E))\simeq \map(\C\otimes \D,\E).
    $$
\end{thmintro}

The proof of this result will be the subject of \ref{section:comparison}.
As will be explained in \ref{section:cech and sq}, the functor $\Sq$ admits a left adjoint that we will denote by $\Gr$. By the Yoneda lemma, the previous theorem is then formally equivalent to asserting the existence of an equivalence
    $$
    \Gr(\C_h \times \D_v) \simeq \C \otimes \D
    $$
that is natural in $(\infty,2)$-categories $\C$ and $\D$.

It should be noted that the original conjecture is slightly different as it stipulates not only that there exists such an equivalence, but that it is realized through the comparison map defined by Gaitsgory and Rozenblyum in \cite[Subsection 10.4.5]{GR}. However, we will show in \ref{prop:auto of gray} that any natural transformation $(-)\otimes(-)\to (-)\otimes(-)$ must be the identity. Consequently, the original comparison map constructed by Gaitsgory and Rozenblyum, and the one considered here necessarily coincide.

\subsection*{Conventions} We will make use of the following notation and terminology throughout the article:
\begin{itemize}
    \item We will write $\infty\Cat$ and $\S$ for the $\infty$-categories of $\infty$-categories and spaces (or $\infty$-groupoids).
    \item The functor $\tau_{0}:\infty\Cat\to \S$ denotes the functor that carries an $\infty$-category to its underlying $\infty$-groupoid, i.e.\ the right adjoint to the inclusion $\S \subset \infty\Cat$. We will write $\tau_{0}^i:\infty\Cat\to \S$ for the left adjoint to this inclusion, called the \textit{groupoidification functor}.
    \item  If $\C$ is an $\infty$-category, we will write $\PSh(\C) := \fun(\C^\op, \S)$ for the $\infty$-category of presheaves on $\C$, and $\PSh_{\mathrm{Set}}(\C):= \fun(\C^\op, \mathrm{Set}) \subset \PSh(\C)$ for the full subcategory of set-valued presheaves on $\C$.
    \item Throughout the text, we will use the \textit{oplax} variant of the Gray tensor product for 2- and $(\infty,2)$-categories. We recall the definitions in \ref{section:prelims}.
    \item We will say that a a subcategory $\C$ of an $\infty$-category $\D$ is \textit{dense} when the induced nerve functor $\D\to \PSh(\C)$, obtained by restricting the Yoneda embedding on $\D$, is fully faithful. In all the situations that we will encounter, $\C$ is a small and $\D$ is cocomplete. In such case, this implies that $\D$ is a reflective subcategory of $\PSh(\C)$, and that any object $x\in\D$ can be canonically expressed as the colimit of the functor $\C_{/x}\to \C \to \D$.
\end{itemize}

\section{Two-dimensional $\infty$-categories}\label{section:prelims}

We will commence by briefly collecting the basic definitions and aspects of $(\infty,2)$-categories and double $\infty$-categories.

\subsection{Shapes for $(\infty,2)$-categories} We recall that, on account of \cite{JoyalTierney}, the full inclusion of the subcategory $\Delta \to \infty\Cat$ is dense. The essential image of the induced inclusion $\infty\Cat \to \PSh(\Delta)$ selects the \textit{Segal spaces} that are \textit{complete} in the sense of Rezk \cite{RezkSeg}. We collect these two notions here:

\begin{definition}\label{def:segal spaces}
    A presheaf $\C : \Delta^\op \to \S$ is called a Segal space if for every $n \geq 0$, the canonical map 
    $$
    \C([n]) \to \C([1]) \times_{\C([0]} \dotsb \times_{\C([0])} \C([1])
    $$
    is an equivalence. Let $J$ be the simplicial set defined by the pushout square 
    \[ \begin{tikzcd}
        {[1] \sqcup [1]} \arrow[rr, "{(\{0 \leq 2\}, (\{1\leq 3\})}"]\arrow[d] && {[3]} \arrow[d] \\
        {[0] \sqcup [0]} \arrow[rr] && J
    \end{tikzcd}
    \] 
    in $\PSh(\Delta)$.
    The Segal space $\C$ is said to be \textit{complete} if it is local with respect to the map $J \to [0]$, i.e. $\C([0]) = \map_{\PSh(\Delta)}([0], \C) \to \map_{\PSh(\Delta)}(J, \C)$ is an equivalence.
\end{definition} 

\begin{remark}
    Let $\C$ be a Segal space.
    The inclusion $[1] \xrightarrow{\{0\leq 1\}} [3] \to J$ induces a monomorphism $\map_{\PSh(\Delta)}(J,\C) \to \map_{\PSh(\Delta)}([1], \C) = \C([1])$ that selects those arrows that are equivalences. This is the content of \cite[Theorem 6.2]{RezkSeg}.
\end{remark}

We may similarly define $(\infty,2)$-categories as presheaves on a suitable collection of 2-categorical shapes. 

\begin{definition}
A (strict) 2-category is called \textit{gaunt} if the only invertible 1- and 2-cells are identities. We will write $2\Gaunt$ for the full subcategory of the category of strict 2-categories spanned by the gaunt 2-categories.
\end{definition}

Important examples of gaunt 2-categories are Joyal's \textit{globular sums}:

\begin{definition}\label{def:glob sums}
Suppose that $n$ is a non-negative integer. Consider a tuple $\overline{m} = (m_0,...,m_{n-1})$ of non-negative integers. Then we will write
$[n;\overline{m}]$ for the gaunt $2$-category whose objects are given by $0, 1, ..., n$ and such that for any $0\leq i,j \leq n$, 
$$[n;\overline{m}](i,j):=\begin{cases}
[m_i]\times \dotsb \times [m_{j-1}] & \text{if $i \leq j$,} \\
\emptyset & \text{otherwise.}
\end{cases}$$
We will write $\Theta_2$ for the full subcategory of $2\Gaunt$ whose objects are of the shape $[n;\overline{m}]$ for some $n$ and tuple $m$. 
The objects in $\Theta_2$ are referred to as \textit{globular sums}. Note that we may view $\Delta$ as a full subcategory of $\Theta_2$ via the functor
$\Delta\to \Theta_2$ that carries $[n]$ to the globular sum $[n;(0,...,0)]$.
\end{definition}

\begin{definition}
A presheaf $\C :\Theta_2^\op \to \S$ is called an \textit{$(\infty,2)$-category} if:
\begin{itemize}
    \item for any tuple $\overline{m}=(m_0,...,m_{n-1})$, the canonical map 
    $$\C([n;\overline{m}])\to \C([1;m_0])\times_{\C([0])}...\times_{\C([0])}\C([1;m_{n-1}])$$ is an equivalence,
    \item for any $n \geq 0$, the canonical map 
    $$\C([1;n])\to \C([1;1])\times_{\C([1])}...\times_{\C([1])}\C([1;1])$$
    is an equivalence,
    \item $\C$ is local with respect to the maps $J \to [0]$ and $[1;J] \to [1]$, where $[1;J]$ is defined by the pushout square
    \[ \begin{tikzcd}
        {[1;1] \sqcup [1;1]} \arrow[rr, "{[1;(\{0 \leq 2\}, (\{1\leq 3\})]}"]\arrow[d] && {[1;3]} \arrow[d] \\
        {[1] \sqcup [1]} \arrow[rr] && {[1;J]}
    \end{tikzcd}
    \] 
    in $\PSh(\Theta_2)$.
\end{itemize}
We will write $(\infty,2)\Cat \subset \PSh(\Theta_2)$ for the full subcategory spanned by the $(\infty,2)$-categories. 
\end{definition}

\begin{construction}
The involutions $\op : 2\Gaunt \to 2\Gaunt$ and $\co : 2\Gaunt \to 2\Gaunt$ given by reversing the directions of 1- and 2-cells respectively, both restrict to involutions $\op, \co : \Theta_2 \to \Theta_2$ of $\Theta_2$. In turn, these induce functors $\op^*, \co^* : \PSh(\Theta_2) \to \PSh(\Theta_2)$ that restrict to involutions $$(-)^\op, (-)^\co : (\infty,2)\Cat \to (\infty,2)\Cat,$$
respectively.
\end{construction}

\begin{notation}
We follow the notation in \cite[Definition 1.1.5]{Effectivity}. There is a canonical inclusion 
$\iota : \infty\Cat \rightarrow (\infty,2)\Cat$ that is obtained by left Kan extending the functor $\Delta \rightarrow (\infty,2)\Cat$ along the inclusion $\Delta \to \infty\Cat$. We will leave the notation of $\iota$ implicit throughout this paper. The functor $\iota$ admits a right adjoint $$\tau_{1}:(\infty,2)\Cat\to \infty\Cat,$$ called the \textit{$1$-truncation} or \textit{$1$-core} that forgets the $2$-cells. Moreover, $\iota$ admits a left adjoint $$\tau^i_{ 1}:(\infty,2)\Cat\to \infty\Cat$$ that is called the \textit{intelligent $1$-truncation}, and inverts the $2$-cells.
We will simply write $\tau_{0}$ and $\tau^i_{0}$ for the functors $(\infty,2)\Cat\to \S$ given by the composites $\tau_{0}\tau_{1}$ and $\tau^i_{0}\tau^i_{ 1}$ respectively.
\end{notation}

\begin{construction}\label{def:discrete}
The canonical inclusion 
$\Theta_2\to 2\Gaunt$ gives rise to an adjunction
\[
	{\Str:(\infty,2)\Cat} \rightleftarrows {2\Gaunt:N}.
\]
The resulting functor $N$ is fully faithful, and its essential image is spanned by the \textit{discrete} $(\infty,2)$-categories, i.e.\ those $(\infty,2)$-categories $\C$ for which $\C([n;\overline{m}])$ is a set for all $[n;\overline{m}] \in \Theta_2$. We will leave the notation of $N$ implicit in what follows.
\end{construction}

\begin{construction}
For the sake of simplicity, we will write $[n;m]$ for the globular sum $[n;(m,...,m)]$. We will write $[\Delta;\Delta]$ for the 1-category that fits in the pushout square
\[\begin{tikzcd}
	{\{[0]\}\times \Delta} & {\Delta\times \Delta} \\
	{\{[0]\}} & {[\Delta;\Delta]}
	\arrow[from=1-1, to=1-2]
	\arrow[from=1-1, to=2-1]
	\arrow[from=1-2, to=2-2]
	\arrow[from=2-1, to=2-2]
\end{tikzcd}\]
of strict categories, which is in fact also a pushout square of $\infty$-categories.
One readily verifies that the bicosimplicial object $\Delta^{\times 2} \rightarrow 2\Gaunt : ([n],[m]) \mapsto [n;m]$ factors to a functor $[\Delta;\Delta]\to 2\Gaunt$.
By \cite[Remark 1.1.10]{Effectivity} or \cite[Remark 3.10]{CompJaco}, the (non-full) subcategories
$$[\Delta;\Delta] \to 2\Gaunt \quad \mbox{and} \quad [\Delta;\Delta]\to 2\Gaunt \to  (\infty,2)\Cat$$
are dense.
\end{construction}

\subsection{The Gray tensor product for gaunt 2-categories} In this subsection, we gather the basics of the oplax Gray tensor product for gaunt $2$-categories.

\begin{notation}
We will write $$(-) \otimes_G (-):2\Gaunt\times 2\Gaunt \to 2\Gaunt$$ for the oplax Gray tensor product of gaunt 2-categories. This operation was originally defined by Gray in \cite{gray2006formal}.  
\end{notation}

\begin{remark}
Given two gaunt 2-categories $A$ and $B$, $A \otimes_G B$ is generated by $k$-cells of shape $a\otimes_G b$ where $a$ is a $(k-i)$-cell of $A$ and $b$ is an $i$-cell of $B$. The composition of these cells is subject to several coherence equations, as described, for example, in \cite[Section 3.1]{maehara2021gray}.
\end{remark}

\begin{remark}
\label{rem:when factor throug tau1}
A functor $\phi:A\otimes_GB\to E$ between gaunt 2-categories factors through $A \otimes_G\tau^i_1B$ if and only if for any $0$-cell $a$ in $A$ and any $2$-cell $b$ in $B$, $\phi(a\otimes_G b)$ is an identity 2-cell.
\end{remark}

\begin{proposition}
\label{prop:auto of gray gaunt}
The functor $\otimes_G$, viewed as an object of  $\fun(2\Gaunt\times 2\Gaunt,2\Gaunt)$, admits no non-trivial endomorphism.
\end{proposition}
\begin{proof}
Let $\phi:\otimes_G\to \otimes_G$ be an endomorphism. We will show that for any pair of gaunt $2$-categories $A$ and $B$, the component $\phi_{A,B}:A\otimes_G B\to A \otimes_G B$ is given by the identity.

\textit{The functor $\phi_{A,B}$ is the identity on objects.} Note that $[0] \otimes_G [0] = [0]$, so that $\phi_{[0],[0]}= \id_{[0]}$. Suppose that we have objects $a\in A$, $b\in B$, then commutativity of the naturality square
\[\begin{tikzcd}
	{[0]\otimes_G[0]} & {A\otimes_G B} \\
	{[0]\otimes_G[0]} & {A\otimes_G B}
	\arrow["{a\otimes_G b}", from=1-1, to=1-2]
	\arrow[equals, from=1-1, to=2-1]
	\arrow["{\phi_{A,B}}", from=1-2, to=2-2]
	\arrow["{a\otimes_G b}"', from=2-1, to=2-2]
\end{tikzcd}\]
implies that $\phi(a\otimes_G b)=a\otimes_G b$.

\textit{The functor $\phi_{A,B}$ is the identity on $1$-cells.} Note that $[1] \otimes_G [0] = [1]$. Hence $\phi_{[1], [0]}$ is an endofunctor of $[1]$ that acts as the identity on objects on account of the previous step. Since $[1]$ is a poset, this implies that $\phi_{[1],[0]}$ is the identity. Similarly, one deduces that $\phi_{[0],[1]}$ is the identity.
Suppose now that we are given a $1$-cell $f : a \to a'\in A$, and an object $b\in B$.
The commutativity of the naturality square
\[\begin{tikzcd}
	{[1]\otimes_G[0]} & {A\otimes_G B} \\
	{[1]\otimes_G[0]} & {A\otimes_G B}
	\arrow["{f\otimes_G b}", from=1-1, to=1-2]
	\arrow[equals, from=1-1, to=2-1]
	\arrow["{\phi_{A,B}}", from=1-2, to=2-2]
	\arrow["{f\otimes_G b}"', from=2-1, to=2-2]
\end{tikzcd}\]
implies that $\phi(f\otimes_G b)=f\otimes_G b$. We can show similarly that $\phi(a\otimes_G g)=a\otimes_G g$ for any object $a \in A$ and any $1$-cell $g: b\to b' \in B$. 
As every $1$-cell of $A\otimes_G B$ can be expressed as a composite of $1$-cells of the shape $f\otimes_G b$ or $a\otimes_G g$, this implies that $\tau_1\phi_{A,B}:\tau_1(A\otimes_G B)\to \tau_1(A\otimes_G B)$ is the identity.

\textit{The functor $\phi_{A,B}$ is the identity on $2$-cells.} We now remark that the gaunt 2-categories $[1]\otimes_G [1], [1;1] \otimes_G [0]$ and $[0]\otimes_G [1;1]$ share the following property: for any pair of parallel $1$-cells, there exists at most one $2$-cell between them. 
As we already established that the components of $\phi$ are the identity on $1$-cells, this implies that $\phi_{[1],[1]}$, $\phi_{[1;1],[0]}$
and $\phi_{[0],[1;1]}$ are identities. Using the same naturality argument as before, this implies that $\phi_{A,B}(f\otimes_G g)=f\otimes_G g$ where $f$ is a $k$-cell of $A$ and $g$ a $(2-k)$-cell of $B$. The natural transformation $\phi_{A,B}$ then acts as the identity on the generating $2$-cells of $A\otimes_G B$, and therefore must be the identity on $2$-cells as well.
\end{proof}

\subsection{The Gray tensor product for $(\infty,2)$-categories} 
We will follow here the approach of Maehara \cite{maehara2021gray}, but adapted to space-valued presheaves instead of set-valued presheaves on $\Theta_2$. We refer to  \ref{subsection:introduction on Gray} of the introduction for an overview of the different definitions of this operations, and the links they maintain.

\begin{construction}[Maehara]The functor
$$
\Theta_2 \times \Theta_2 \to (\infty,2)\Cat : (a,b) \mapsto a \otimes_G b
$$
extends to a unique functor $\PSh(\Theta_2) \times \PSh(\Theta_2) \rightarrow (\infty,2)\Cat$ that preserves colimits in both variables. This then restricts to a functor
$$(-)\otimes(-):(\infty,2)\Cat\times (\infty,2)\Cat \to (\infty,2)\Cat$$
that will be called the \textit{(oplax) Gray tensor product}.
\end{construction}

\begin{theorem}[Maehara]
\label{theo:Gray tensor product preserves colimits}
The Gray tensor product preserves colimits in both variables.
\end{theorem}

\begin{proof}
The extension $\PSh(\Theta_2) \times \PSh(\Theta_2) \rightarrow (\infty,2)\Cat$ preserves colimits in both variables and sends each spine inclusion and completeness extension to an equivalence on account of \cite[Subsection 6]{maehara2021gray}.
\end{proof}

\begin{theorem}[Maehara]
\label{theo:Gray tensor of globular sums}
The canonical map
$$a_0\otimes a_1\otimes...\otimes a_n\to a_0\otimes_Ga_1\otimes_G ...\otimes_G a_n$$
is an equivalence for every sequence $a_0, a_1,...,a_n$ of globular sums.
\end{theorem}

\begin{proof}
This is the content of  \cite[Corollary 7.11]{maehara2021gray}.
\end{proof}

\begin{corollary} 
The Gray tensor product extends to a monoidal structure on $(\infty,2)\Cat$ such that $\Str:(\infty,2)\Cat\to 2\Gaunt$ is a (strong) monoidal functor.
\end{corollary}

\begin{proof}
Let $\left<\Theta_2 \right>$ be the smallest full subcategory of $2\Gaunt$ that includes $\Theta_2$ and that is stable under the Gray tensor product $\otimes_G$. As $\left<\Theta_2 \right>$ includes $\Theta_2$, it is again a dense subcategory of $(\infty,2)\Cat$. The associated fully faithful functor $(\infty,2)\Cat \rightarrow \PSh(\left<\Theta_2 \right>)$ given by restricting the Yoneda embedding admits a left adjoint
$$L:\PSh(\left<\Theta_2 \right>)\to (\infty,2)\Cat$$ that is given by $L'i^*$ where $L'$ is the left adjoint to the inclusion $(\infty,2)\Cat \rightarrow \PSh(\Theta_2)$.

By Day convolution (\cite[Proposition 4.8.1.10]{LurieHigherAlgebra} or \cite[Proposition 2.14]{glasman2016day}), the monoidal structure on $\left<\Theta_2\right>$ induces a monoidal structure on $\PSh(\left<\Theta_2\right>)$ whose tensor product will be denoted by $\widehat{\otimes}_G$. It has the property that  $\widehat{\otimes}_G$ commutes with colimits in both variables, and the Yoneda embedding $\left<\Theta_2\right> \to \PSh(\left<\Theta_2\right>)$ upgrades to a monoidal functor. If $X, Y \in \PSh(\left<\Theta_2\right>)$, then there is a natural equivalence $L(X \widehat{\otimes}_G Y) \simeq LX \otimes LY$ in $(\infty,2)\Cat$. This follows from the observation that both sides commute with colimits in the variables $X$ and $Y$, so that we may reduce checking for $X, Y \in \left< \Theta_2\right>$ by density. The desired conclusion then follows from \ref{theo:Gray tensor of globular sums}. We may  then use \ref{theo:Gray tensor product preserves colimits} to deduce that the localization functor $L$ is compatible with the monoidal structure on $\PSh(\left< \Theta_2 \right>)$ in the sense of \cite[Definition 2.2.1.6]{LurieHigherAlgebra}. Thus \cite[Proposition 2.2.1.9]{LurieHigherAlgebra} implies that the monoidal structure on $\PSh(\left<\Theta_2\right >)$ transfers to a monoidal structure on $(\infty,2)\Cat$ so that $L$ is monoidal.

To see that $\Str$ is monoidal, note that the monoidal inclusion $\left<\Theta_2\right> \to 2\Gaunt$ extends to a monoidal functor $\PSh(\left<\Theta_2\right>) \to 2\Gaunt$ by the universal property of Day convolution. This then derives to a monoidal functor $(\infty,2)\Cat \to 2\Gaunt$ whose underlying functor is precisely $\Str$.
\end{proof}

\begin{proposition}
\label{prop:auto of gray}
The functor $\otimes$, viewed as an object of $\fun((\infty,2)\Cat\times (\infty,2)\Cat,(\infty,2)\Cat)$, has no non-trivial endomorphism.
\end{proposition}
\begin{proof}  
As $\Theta_2$ is a dense subcategory of $2\Gaunt$ and $(\infty,2)\Cat$, the two functors
\begin{gather*}
\fun(2\Gaunt\times 2\Gaunt,2\Gaunt) \to \fun(\Theta_2 \times \Theta_2, 2\Gaunt), \\
\fun((\infty,2)\Cat\times (\infty,2)\Cat,(\infty,2)\Cat) \to \fun(\Theta_2 \times \Theta_2, (\infty,2)\Cat)
\end{gather*}
given by restriction are both fully faithful when restricted to functors that are cocontinuous in each variable.
Since the inclusion $2\Gaunt\to (\infty,2)\Cat$ is fully faithful, we obtain the following equivalences between mapping spaces:
\begin{align*}
\map_{\fun((\infty,2)\Cat^{\times 2},(\infty,2)\Cat)}(\otimes, \otimes) &\simeq \map_{\fun(\Theta_2^{\times 2}, (\infty,2)\Cat)}(\otimes|\Theta_2^{\times 2},\otimes|\Theta_2^{\times 2}) \\
&\simeq \map_{\fun(\Theta_2^{\times 2},2\Gaunt)}(\otimes_G|\Theta_2^{\times 2},\otimes_G|\Theta_2^{\times 2}) \\
&\simeq \map_{\fun(2\Gaunt^{\times 2},2\Gaunt)}(\otimes_G,\otimes_G).
\end{align*}
Hence  \ref{prop:auto of gray gaunt} implies that the space of endomorphisms of $\otimes$ is contractible.
\end{proof}

We record the following facts that will be of use later:

\begin{proposition}
For any two $(\infty,2)$-categories $\C$ and $\D$, there are natural equivalences
$$(\C\otimes \D)^{\op}\simeq \D^{\op}\otimes \C^{\op} ~~ \text{and} ~~ (\C\otimes \D)^\co \simeq \D^\co\otimes \C^\co.$$
\end{proposition}
\begin{proof}
We may reduce to the case that $\C$ and $\D$ are globular sums. As the Gray tensor product of globular sums is gaunt, the result follows from \cite[Proposition A.20]{ara2020joint}.
\end{proof}

\begin{lemma}
\label{lemma:square as a colimit}
The gaunt $2$-category $[1]\otimes [1]$ is equivalent to the colimit of the diagram
\[\begin{tikzcd}
	{[2]} & {[1;\{0\}]} & {[1;1]} & {[1;\{1\}]} & {[2]}
	\arrow["{d_1}"', from=1-2, to=1-1]
	\arrow[from=1-2, to=1-3]
	\arrow[from=1-4, to=1-3]
	\arrow["{d_1}", from=1-4, to=1-5]
\end{tikzcd}\]
computed in $\PSh(\Theta_2)$.
\end{lemma}
\begin{proof}
On account of \cite[Lemma 2.1.1.5]{loubaton2024categorical}, we may compute this colimit in $\PSh_\mathrm{Set}(\Theta_2)$, where it is readily verified.
\end{proof}

\begin{proposition}
\label{lemma:globular sum as quotient of Gray}
Let $n,m$ be non-negative integers. Then there are natural pushout squares
\[\begin{tikzcd}
	{\coprod_{k\leq n}\{k\}\otimes[m]} & {[n]\otimes[m]} & {\coprod_{k\leq n}[m]^\op\otimes \{k\}} & {[m]^\op\otimes[n]} \\
	{\coprod_{k\leq n}\{k\}\otimes[0]} & {[n;m]}, & {\coprod_{k\leq n}[0]\otimes \{k\}} & {[n;m]}
	\arrow[from=1-1, to=1-2]
	\arrow[from=1-1, to=2-1]
	\arrow[from=1-2, to=2-2]
	\arrow[from=1-3, to=1-4]
	\arrow[from=1-3, to=2-3]
	\arrow[from=1-4, to=2-4]
	\arrow[from=2-1, to=2-2]
	\arrow[from=2-3, to=2-4]
\end{tikzcd}\]
in $(\infty,2)\Cat$.
\end{proposition}

\begin{proof}
We focus on the left-hand square, the second is obtained by applying the duality $(\uvar)^\co$. 
The right vertical arrow $[n]\otimes[m]\to [n;m]$ was constructed in  \cite[Construction 3.57]{CompJaco}. Since the vertices of the squares are parametrized by functors $\Delta \times \Delta \to (\infty,2)\Cat$ that are cocontinuous in each variable, we can reduce to the case that $n=1$ and $m=1$. This follows from \ref{lemma:square as a colimit}.
\end{proof}

As promised, we conclude this section by proving the equivalence between the Gray tensor product $\otimes_L$ and $\otimes$. 

\begin{proposition}
\label{prop:comparaison of gray tensor product}
There is a unique invertible natural transformation $\otimes_L\to \otimes$.
\end{proposition}

\begin{proof}
We recall that the tensor product $\otimes_L$ is obtained as the composite
$$(\infty,2)\Cat\times(\infty,2)\Cat\to (\infty,\omega)\Cat\times(\infty,\omega)\Cat\xrightarrow{\otimes^\omega_L} (\infty,\omega)\Cat\xrightarrow{\tau_2^i} (\infty,2)\Cat, $$
where $\otimes_L^{\omega}$ denotes the Gray tensor product on the $\infty$-category of $(\infty,\omega)$-categories which was defined in \cite[Construction 1.4.1]{Effectivity}. Here $\tau^i_2$ is the intelligent $2$-truncation functor of \cite[Definition 1.1.5]{Effectivity}.

The unicity of the comparison will follow from \ref{prop:auto of gray}. It remains to show the existence.
On account of \cite[Theorem 1.4.14]{Effectivity}, for any globular sums $a,b$ in $\Theta_2$,  $a\otimes_L^{\omega}b$ coincides with $a\otimes_G^\omega b$, where $\otimes_G^\omega$ denotes the Gray tensor product for gaunt $\omega$-categories. Recall that the Gray tensor product on gaunt $2$-categories is the $2$-truncation of the Gray tensor product on gaunt $\omega$-categories (see Appendix A of \cite{ara2020joint}). 

By \ref{theo:Gray tensor of globular sums}, this then induces a comparison functor $\otimes_L\to \otimes$. As these two bifunctors preserve colimits in both variables, it is sufficient to show that for any $0\leq n,m$, the component $\tau_2^i([1;n]\otimes^{\omega}_L[1;m])\to [1;n]\otimes[1;m]$ is an equivalence. Since the $2$-truncation commutes with the strictification functor, it thus suffices to demonstrate that $\tau_2^i([1;n]\otimes^{\omega}_L[1;m])$ is a gaunt $2$-category.

It follows from \cite[Lemma 1.6.12]{Effectivity} that $\tau^i_1(A \otimes_L^{\omega} B)\simeq \tau^i_1(A)\times \tau^{i}_1(B)$ for all $(\infty,\omega)$-categories $A$ and $B$. Combining this with \cite[Proposition 1.4.22]{Effectivity}, this then implies that $\tau_2^i([1;m]\otimes^{\omega}_L[1;n])$ is canonically the colimit of the diagram
\[\begin{tikzcd}
	{[1;[n]\times\{0\}\times [m]]} && {[1;[n]\times\{1\}\times [m]]} \\
	{[2;n,m]} & {[1;[n]\times[1]\times [m]]} & {[2;m,n]}
	\arrow[from=1-1, to=2-1]
	\arrow[from=1-1, to=2-2]
	\arrow[from=1-3, to=2-2]
	\arrow[from=1-3, to=2-3]
\end{tikzcd}\]
of gaunt 2-categories in $(\infty,2)\Cat$. 
Let $F_{n,m} : M \to 2\Gaunt$ be the functor that shapes the above diagram.
On account of \ref{lemma:square as a colimit}, there is an equivalence 
$$
\tau_2^i([1]\otimes^{\omega}_L[1])\simeq \colim(M \xrightarrow{F_{0,0}} 2\Gaunt \to (\infty,2)\Cat \to \PSh(\Theta_2)).
$$
in $\PSh(\Theta_2)$. As the natural transformation $F_{n,m}\to F_{0,0}$ is cartesian, \cite[Proposition 2.2.1.27]{loubaton2024categorical} implies that there is an equivalence 
$$
\tau_2^i([1;n]\otimes^{\omega}_L[1;m])\simeq \colim(M \xrightarrow{F_{n,m}} 2\Gaunt \to (\infty,2)\Cat \to \PSh(\Theta_2))
$$
in $\PSh(\Theta_2)$.
The composite $2\Gaunt \to (\infty,2)\Cat \to \PSh(\Theta_2)$ factors  through the inclusion $\PSh_{\mathrm{Set}}(\Theta_2) \to \PSh(\Theta_2)$ that is induced by the functor $\mathrm{Set} \to \S$; see \ref{def:discrete}. Now \cite[Lemma 2.1.1.5]{loubaton2024categorical} implies that
\begin{align*}
\tau_2^i([1;m]\otimes^{\omega}_L[1;n]) &\simeq \colim(M \xrightarrow{F} 2\Gaunt \to \PSh_{\mathrm{Set}}(\Theta_2) \to \PSh(\Theta_2)) 
\\&\simeq \colim(M \xrightarrow{F} 2\Gaunt \to \PSh_{\mathrm{Set}}(\Theta_2)) 
\end{align*}
in $\PSh(\Theta_2)$. So $\tau_2^i([1;m]\otimes^{\omega}_L[1;n])$ is discrete, and thus gaunt.
\end{proof}

\subsection{Double $\infty$-categories}\label{subsection:dbl cats} We now recall the necessary material on double $\infty$-categories.

\begin{definition}
    A bisimplicial space $\P : \Delta^\op \times \Delta^\op \to \S$ is called a \textit{double $\infty$-category} if:
    \begin{itemize}
        \item the restrictions $\P([n],-) : \Delta^\op \to \S$ is a complete Segal space for every $n \geq 0$
        \item the restriction $\P(-,[m]) : \Delta^\op \to \S$ is a Segal space for every $m \geq 0$. 
    \end{itemize}
    We will write $\DbliCat \subset \PSh(\Delta^{\times 2})$ for the full subcategory spanned by the double $\infty$-categories.
\end{definition}

\begin{notation}
    The Yoneda embedding $\Delta^{\times 2} \to \PSh(\Delta^{\times 2})$ factors through $\DbliCat$. We will write $\langle n,m \rangle$ for the double $\infty$-category given by the image of $([n],[m])$.
    
    This was previously denoted by $[n,m]$ in \cite{JacoThesis} and \cite{CompJaco}, but we have opted for this notation here to make it better distinguishable from the objects of $[\Delta;\Delta]$.
\end{notation}

\begin{remark}
    Equivalently, a double $\infty$-category $\P$ is the datum of a functor $\P_\bullet :  \Delta^\op \to \infty\Cat$ so that the canonical functor $\P_n \to \P_1 \times_{\P_0} \dotsb \times_{\P_0} \P_1$ is an equivalence, where $\P_n$ is defined by the formula: $$\map_{\infty\Cat}([m], \P_n) =  \map_{\DbliCat}(\langle n,m \rangle, \P).$$
\end{remark}

\begin{definition}
   A double $\infty$-category $\P$ is called \textit{locally complete} if for every two objects $x,y \in \P$, the Segal space  
   $
   \P(-,[1]) \times_{\P(-,[0]) \times \P(-,[0])} \{(x,y)\}
   $
   is complete. More strongly, we will say that $\P$ is \textit{complete} if $\P(-,[1])$ is complete. We will write $\CDbliCat \subset \DbliCat$ for the full subcategory spanned by the complete double $\infty$-categories. 
\end{definition}

\begin{remark}
    The subcategory $\CDbliCat \subset \DbliCat$ is reflective. 
\end{remark}

\begin{remark}
    The $\infty$-categories $\DbliCat$ and $\CDbliCat$ (and also, the full subcategory spanned by the locally complete double $\infty$-categories) inherit the cartesian product from $\PSh(\Delta^{\times 2})$, and they are cartesian closed on account of \cite[Proposition 3.44]{CompJaco}.
\end{remark}

\begin{construction}\label{cons:dualities for double}
    We recall the following duality operations from \cite[Section 3]{CompJaco}.
There are involutions $\hop := \op \times \id, \vop := \id \times \op :\Delta\times \Delta \to \Delta \times \Delta$  that induce functors $\hop^*, \vop^* : \PSh(\Delta\times \Delta) \to \PSh(\Delta\times \Delta)$ and restrict to the \textit{horizontal} and \textit{vertical opposite} involutions $$(-)^\hop, (-)^\vop :\DbliCat\to \DbliCat,$$
respectively. 
We will also consider the \textit{transpose} duality. This is defined similarly. The functor $\tp : \Delta \times \Delta \to \Delta \times \Delta$ that swaps the coordinates, induces a functor $\tp^* : \PSh(\Delta\times \Delta)\to\PSh(\Delta\times\Delta)$ that restricts to an involution
$$
(-)^\tp : \CDbliCat \to \CDbliCat.
$$
We note that the strong completeness assumption is important here.
\end{construction}

\begin{notation}\label{not: vertical horizontal inclusion}
    We recall from \cite[Subsection 3.3]{CompJaco} that there are fully faithful functors
    $
    (-)_h, (-)_v : \PSh([\Delta;\Delta]) \to \PSh(\Delta^{\times 2})
    $
    that preserve colimits, and are determined by the fact that there are natural pushout squares 
    \[
    \begin{tikzcd}
	{\coprod_{k=0}^n\langle 0,m \rangle} & {\langle n,m \rangle} & {\coprod_{k\leq n}\langle m, 0 \rangle}^\hop & {\langle m,n\rangle}^\hop \\
	{\coprod_{k=0}^n\langle 0, 0 \rangle} & {[n;m]_h} & {\coprod_{k\leq n}\langle 0, 0\rangle} & {[n;m]_v}.
	\arrow[from=1-1, to=1-2]
	\arrow[from=1-1, to=2-1]
	\arrow[from=1-2, to=2-2]
	\arrow[from=1-3, to=1-4]
	\arrow[from=1-3, to=2-3]
	\arrow[from=1-4, to=2-4]
	\arrow[from=2-1, to=2-2]
	\arrow[from=2-3, to=2-4]
	\end{tikzcd}
	\]
	Since $[\Delta;\Delta]$ is dense in $(\infty,2)\Cat$, we may restrict these functors to $(\infty,2)$-categories. It follows from \cite[Subsection 3.4]{CompJaco} 
	that the functors  restrict to fully faithful and colimit preserving functors $$(-)_h, (-)_v : (\infty,2)\Cat \to \CDbliCat.$$
	These are called the \textit{horizontal} and \textit{vertical inclusion} functors respectively.
\end{notation}

\section{The directed \v{C}ech nerve and the squares construction}\label{section:cech and sq}

Given a morphism $f:A\to B$ between spaces, its \v{C}ech nerve, denoted by $\Cn(f)$, is defined to be the simplicial space that in degree $n$ corresponds to the $n$-fold iterated pullback
$$\Cn(f)_n = A\times_BA\times_B\dotsb \times_BA\;$$
see \cite[Subsection 6.1.2]{HTT}.
We will focus here on the categorification of this notion of \v{C}ech nerve that was introduced by the first author in \cite{Effectivity}. As a special case, we will recover the \textit{squares construction} that plays a fundamental role in the $(\infty,2)$-categorical set up of Gaitsgory and Rozenblyum \cite[Section 10.4]{GR}.

\begin{definition}
We will consider the full subcategory
$
\Filt \subset \fun([1], (\infty,2)\Cat)
$
of \textit{(1,1)-filtrations}, or simply, \textit{filtrations}, that is spanned by those functors $\C \to \D$ so that $\C$ is a $(\infty,1)$-category.

We will write 
$
\Filt^\twoheadrightarrow \subset \Filt
$ 
for the full subcategory spanned by the \textit{eso} (\textit{essentially surjective on objects}) filtrations $\C \to \D$, i.e. for which $\C \to \tau_{1}\D$ is essentially surjective.
\end{definition}

\begin{remark}\label{rem: trivial filtration}
    The functor $\Filt \to (\infty,2)\Cat$ given by evaluation at $1$ admits a right adjoint that carries an $(\infty,2)$-category $\C$ to the filtration $\tau_1\C \to \C$.
\end{remark}

\begin{construction}\label{cons: realization-nerve adjunction}
Let us consider the bicosimplicial filtration
$$
\textstyle \Delta \times \Delta \rightarrow \Filt : ([n],[m]) \mapsto  (\tau_0[n] \otimes [m] \rightarrow [n] \otimes [m]).
$$
The $\infty$-category $\Filt$ is cocomplete, and hence we obtain an  adjunction
$
 \PSh(\Delta \times \Delta) \rightleftarrows \Filt : \Cn(-)
$
so that the left functor is Kan extended from the bicosimplicial filtration.
For a filtration $f : \C \rightarrow \D$, $\Cn(f)$ is computed level-wise by 
$$
\textstyle \Cn(f)_{n,m} = \map_\Filt(\tau_0[n] \otimes [m] \rightarrow [n] \otimes [m], \C \xrightarrow{f} \D),
$$
and thus sits in a pullback square \[\begin{tikzcd}[ampersand replacement=\&]
	{\Cn(f)_{n,m}} \& \map_{(\infty,2)\Cat}([n] \otimes [m], \D) \\
	\prod_{k=0}^n\map_{(\infty,2)\Cat}(\{k\} \otimes [m], \C) \& \prod_{k=0}^n\map_{(\infty,2)\Cat}(\{k\} \otimes [m], \D).
	\arrow[from=1-1, to=1-2]
	\arrow[from=1-1, to=2-1]
	\arrow[from=1-2, to=2-2]
	\arrow[from=2-1, to=2-2]
\end{tikzcd}\]
Since the Gray tensor product is cocontinuous in both variables, $\Cn(f)$ is a double $\infty$-category, and the above adjunction derives 
to an adjunction
$$
|-| :  \DbliCat \rightleftarrows \Filt : \Cn(-).
$$
If $f$ is a filtration, then the double $\infty$-category $\Cn(f)$ is called the \textit{directed \v{C}ech nerve of $f$}. If $\P$ is a double $\infty$-category, then $|\P|$ is called the \textit{realization of $\P$}. The domain of $|\P|$ is always given by the vertical $\infty$-category $\P_0$. 

After composing the above realization-nerve adjunction with the adjunction of \ref{rem: trivial filtration}, we obtain another adjunction 
$$
\Gr : \DbliCat \rightleftarrows (\infty,2)\Cat : \Sq.
$$
The left adjoint carries a double $\infty$-category $\P$ to the codomain $|\P|_1$ of its realization. The right adjoint carries an $(\infty,2)$-category $\C$ to its \textit{double $\infty$-category of squares in $\C$}, and it is level-wise described by 
    $$\map_{\DbliCat}(\langle n,m\rangle, \Sq(\C))= \map_{(\infty,2)\Cat}([n] \otimes [m], \C).$$
\end{construction}

\begin{example}
    \label{exemple:realization of vertical and horizontal inclusion}
    Let $\C$ be an $(\infty,2)$-category. In \cite[Construction 3.57]{CompJaco}, canonical inclusions
    $\C_h \to \Sq(\C)$ and $\C_v \to \Sq(\C)$ were constructed. These inclusions are adjoint to equivalences $\Gr(\C_h) \simeq \C$ and $\Gr(\C_v) \simeq \C$ as argued in \cite[Proposition 3.58]{CompJaco}. One can use these observations to show that 
    $|\C_h| \simeq \tau_0\C \to \C$ and  $|\C_v| \simeq \tau_1\C \to \C$.
\end{example}

\begin{remark}
\label{remark:Gr and completion}
It is readily verified that the double $\infty$-category $\Sq(\C)$ is always complete, 
for every $(\infty,2)$-category $\C$. By adjunction, this implies that $\Gr$ factors through the completion functor $\DbliCat \to \CDbliCat$ that is left adjoint to the inclusion. In particular, the restriction $\Gr: \CDbliCat \to (\infty,2)\Cat$ preserves colimits.
\end{remark}

\subsection{The universal property of the \v{C}ech nerve and squares}\label{subsection:uni prop} Using the results of \cite{Effectivity}, we may characterize the directed \v{C}ech nerve in terms of \textit{companions}. This notion makes an appearance in the work of Gaitsgory--Rozenblyum as well \cite[Subsection 10.5.1]{GR}, but under a different name. 

\begin{definition}\label{def:comp}
	Let $f : x \rightarrow y$ be a vertical arrow of a double $\infty$-category $\P$. A horizontal arrow 
	$F : x \rightarrow y$ of $\P$ is called the \textit{companion} of $f$  
	if there exist cells
	\[
		\eta = \begin{tikzcd}
			x\arrow[d, equal] \arrow[r, equal, ""name=f] & x\arrow[d, "f"] \\
			x \arrow[r, "F"'name=t] & y
			\arrow[from=f,to=t,Rightarrow, shorten <= 6pt, shorten >= 6pt]
		\end{tikzcd} 
		\quad 
		\text{and}
		\quad
		\epsilon = \begin{tikzcd}
			x \arrow[r, "F"name=f]\arrow[d, "f"' ] & y \arrow[d, equal] \\
			y \arrow[r, equal, ""'name=t] & y
			\arrow[from=f,to=t,Rightarrow, shorten <= 6pt, shorten >= 6pt]
		\end{tikzcd}
	\] 
	that satisfy the following two \textit{triangle identities}:
	\[
		\begin{tikzcd}
			x\arrow[d, equal] \arrow[r, equal, ""'name=h1] & x\arrow[d, "f"] \\
			x \arrow[r, ""name=h2]\arrow[d, "f"'] & y \arrow[d, equal] \\
			y \arrow[r, equal, ""'name=h3] & y
			\arrow[from=h1, to=h2, phantom, "\scriptstyle\eta"]
			\arrow[from=h2, to=h3, phantom, "\scriptstyle\epsilon"]
		\end{tikzcd}
		\simeq
		\begin{tikzcd}
			x\arrow[r,equal] \arrow[d, "f"'name=t1] & x \arrow[d, "f"name=f1] \\
			y \arrow[r, equal]  & y,
			\arrow[from=f1, to=t1, equal, shorten <= 14pt, shorten >= 14pt]
		\end{tikzcd}
		\quad 
		\begin{tikzcd}
			x \arrow[r, equal, ""name=h1]\arrow[d, equal] & x \arrow[d] \arrow[r, "F"name=h3] & y\arrow[d, equal] \\
			x \arrow[r, "F"'name=h2] & y \arrow[r, equal, ""'name=h4] & y
			\arrow[from=h1, to=h2, phantom, "\scriptstyle\eta"]
			\arrow[from=h3, to=h4, phantom, "\scriptstyle\epsilon"]
		\end{tikzcd}
		\simeq
		\begin{tikzcd}
			x\arrow[d,equal] \arrow[r, "F"name=f1] & y \arrow[d,equal] \\
			x \arrow[r, "F"'name=t1] & y.
			\arrow[from=f1, to=t1, equal, shorten <= 12pt, shorten >= 10pt]
		\end{tikzcd}
	\]
    The cells $\eta$ and $\epsilon$ are called the \textit{companionship unit} and \textit{counit} respectively.
\end{definition}

\begin{remark}
    It is shown by the second author in \cite{CompJaco} that companions are unique up to contractible choice if they exist, and that $\Sq([1])$ is the universal double $\infty$-category that contains a single companionship. Moreover, it is shown that the spaces of companionship units and counits for a vertical arrow with a companion are always contractible.
\end{remark}

\begin{remark}
    Examples of companions may be found in \cite[Section 4]{CompJaco}. As explained in \cite[Example 4.5]{CompJaco}, every vertical arrow in the squares construction $\Sq(\C)$ on an $(\infty,2)$-category $\C$ admits a companion, and each horizontal arrow is a companion. This in particular implies that \v{C}ech nerves of filtrations admit all companions; see also \cite[Lemma 2.6.7]{Effectivity}.
\end{remark}

\begin{remark}
    One may readily verify that if $f:x\to y$ and $g:y\to z$ are vertical arrows  in a double $\infty$-category $\P$ that admit companions $F$ and $G$ respectively, then the composite $GF$ is the companion of $gf$.
\end{remark}

We will show the following characterization:

\begin{theorem}\label{thm:uni prop of cech}
    Let $\P$ be a double $\infty$-category. Then restriction along the unit $\P \to \Cn|\P|$ gives a monomorphism 
    $$
    \map_{\DbliCat}(\Cn|\P|, \Q) \to \map_{\DbliCat}(\P, \Q)
    $$
    for every double $\infty$-category $\Q$. The image is given by the subspace of functors $\P \to \Q$ so that every vertical arrow of $\P$ is carried to a vertical arrow in $\Q$ that admits a companion.
\end{theorem}

The theorem specializes to the following result that appeared as a conjecture in \cite[Section 2.7]{JacoThesis} before, and generalizes \cite[Theorem 4.13]{CompJaco}:

\begin{corollary}\label{cor:uni prop squares vert}
    Let $\C$ be an $(\infty,2)$-category, and $\Q$ be a double $\infty$-category. Then the canonical map $\C_v \to \Sq(\C)$ induces a monomorphism 
     $$
    \map_{\DbliCat}(\Sq(\C), \Q) \to \map_{\DbliCat}(\C_v, \Q)
    $$
    whose image is given by the functors $\C_v \to \Q$ that carry every arrow in $\C$ to a vertical arrow of $\Q$ that admits a companion.
\end{corollary}
\begin{proof}
    On account of the example \ref{exemple:realization of vertical and horizontal inclusion}, $|\C_v|$ corresponds to the filtration $\tau_1\C\to \C$, and we then have $\Cn|\C_v|\simeq \Sq(\C)$. The result then follows from  \ref{thm:uni prop of cech}.
\end{proof}

We will need an auxiliary definition.

\begin{construction}
    Suppose that $\P$ is a double $\infty$-category. Then for non-negative integers $n$ and $m$, we may consider subspaces
    $
    (\tau_\mathrm{vcomp}\P)_{n,m}, (\tau_\mathrm{hcomp})_{n,m} \subset \P_{n,m}
    $
    which are defined so that for every $\sigma : \langle n,m \rangle \to \P$, the following holds:
    \begin{enumerate}
        \item $\sigma$ belongs to $ (\tau_\mathrm{vcomp}\P)_{n,m}$ if and only if for every map $f : [1]_v \to \langle n,m \rangle$, the restriction $\sigma|f$ selects a vertical arrow that admits a companion,
        \item $\sigma$ belongs to $(\tau_\mathrm{hcomp}\P)_{n,m}$ if and only if for every map $f : [1]_h \to \langle n,m \rangle$, the restriction $\sigma|f$ selects a horizontal arrow that is a companion.
    \end{enumerate}
    These subspaces assemble to bisimplicial subspaces $\tau_\mathrm{vcomp}\P$ and $\tau_\mathrm{hcomp}\P$ of $\P$. With these definitions, note that for any bisimplicial space $X$, the maps
    $$
    \map_{\PSh(\Delta \times \Delta)}(X, \tau_\mathrm{vcomp}\P),~\map_{\PSh(\Delta \times \Delta)}(X, \tau_\mathrm{hcomp}\P) \to \map_{\PSh(\Delta \times \Delta)}(X, \P)
    $$
    are monomorphisms whose image correspond to the maps $\sigma : X \to \P$ that carry every vertical arrow (resp.\ horizontal arrow) of $X$ to a vertical arrow (resp.\ horizontal arrow) of $\P$ that admits (resp.\ is given by) a companion.
\end{construction}

\begin{lemma}
    Suppose that $\P$ is a double $\infty$-category. Then $\tau_\mathrm{vcomp}\P$ and $\tau_\mathrm{hcomp}\P$ are double $\infty$-categories.
\end{lemma}
\begin{proof}
    One may readily verify that the Segal conditions in both simplicial directions are met. It remains to check that both objects are local with respect to $[J]_v \to [0]_v$. Let us consider the commutative diagram
     \[
    \begin{tikzcd}[column sep = small, row sep = small]
        \map_{\PSh(\Delta^{\times 2})}([0]_v, \tau_\mathrm{vcomp}\P) \arrow[d, "\simeq"'] \arrow[r] & \map_{\PSh(\Delta^{\times 2})}([J]_v, \tau_\mathrm{vcomp}\P) \arrow[d]  \\
        \map_{\PSh(\Delta^{\times 2})}([0]_v, \P)  \arrow[r, "\simeq"] & \map_{\PSh(\Delta^{\times 2})}([J]_v, \P)   \\ 
        \map_{\PSh(\Delta^{\times 2})}([0]_v,\tau_\mathrm{hcomp}\P) \arrow[r]\arrow[u, "\simeq"] &  \map_{\PSh(\Delta^{\times 2})}([J]_v, \tau_\mathrm{hcomp}\P) \arrow[u, "\simeq"],
    \end{tikzcd}
    \]
    where we marked all the equivalences that follow either by construction or assumption. The top right vertical arrow is a monomorphism by construction, and must be essentially surjective as well. Thus all horizontal arrows are equivalences by 2-out-of-3.
\end{proof}

\begin{proof}[Proof of \ref{thm:uni prop of cech}]
    We equivalently have to show that $$\map_{\DbliCat}(\Cn|\P|, \tau_{\mathrm{vcomp}}\Q) \to \map_{\DbliCat}(\P, \tau_\mathrm{vcomp}\Q)$$
    is an equivalence. On account of \cite[Theorem 3.4.1]{Effectivity}, there exists an eso filtration $f$ so that $\tau_\mathrm{comp}\Q \simeq \Cn(f)$. The total composite in the factorization 
     $$\map_{\Filt}(|\P|, f) \to \map_{\DbliCat}(\Cn|\P|, \Cn(f)) \to \map_{\DbliCat}(\P, \Cn(f))$$
     is an equivalence by adjunction. The first map is an equivalence on account of \cite[Theorem 3.4.1]{Effectivity}, thus the desired result follows from 2-out-of-3.
\end{proof}

There is a version of \ref{cor:uni prop squares vert} where the vertical inclusion is replaced by the horizontal inclusion into the squares construction. In this case, we need an additional completeness assumption (cf.\ \cite[Corollary 4.15]{CompJaco}).

\begin{theorem}\label{thm:uni prop squares hor}
    Let $\C$ be an $(\infty,2)$-category, and $\Q$ be a locally complete double $\infty$-category. Then the canonical map $\C_h \to \Sq(\C)$ induces a monomorphism 
     $$
    \map_{\DbliCat}(\Sq(\C), \Q) \to \map_{\DbliCat}(\C_h, \Q)
    $$
    whose image is given by the functors $\C_h \to \Q$ that carry every arrow in $\C$ to a horizontal arrow of $\Q$ that is a companion.
\end{theorem}

To prove this, we need the following input: 
\begin{lemma}\label{lem:comp core is complete}
    Suppose that $\P$ is a locally complete double $\infty$-category. Then the sub double $\infty$-category $\Q := \tau_\mathrm{vcomp}\P \times_\P \tau_\mathrm{hcomp}\P$ of $\P$ is a complete double $\infty$-category.
\end{lemma}
\begin{proof}
    For brevity, let us write $\Q := \tau_\mathrm{vcomp}\P \times_\P \tau_\mathrm{hcomp}\P$.
    One readily verifies that $\Q$ is again locally complete. It remains to check that it is local with respect to $[J]_h \to [0]_h$. To this end, we may consider the commutative diagram 
    \[
    \begin{tikzcd}[column sep = small, row sep = small]
        \map_{\DbliCat}([0,0], \Q) \arrow[d,equal] \arrow[r] & \map_{\DbliCat}([1]_v, \Q)' \arrow[r, hook] & \map_{\DbliCat}([1]_v, \Q)  \\
        \map_{\DbliCat}([0,0], \Q) \arrow[d,equal] \arrow[r] & \map_{\DbliCat}(\Sq([1]), \Q)' \arrow[r, hook] \arrow[u] \arrow[d] & \map_{\DbliCat}(\Sq([1]), \Q)\arrow[u] \arrow[d]  \\ 
        \map_{\DbliCat}([0,0], \Q) \arrow[r] &  \map_{\DbliCat}([1]_h, \Q)' \arrow[r, hook] & \map_{\DbliCat}([1]_h, \Q).
    \end{tikzcd}
    \] 
    Here, $\map_{\DbliCat}([1]_i, \Q)' \simeq \map_{\DbliCat}(J_i, \Q)$ denotes the subspace of equivalences for $i \in \{v,h\}$. We wrote $\map_{\DbliCat}(\Sq([1]), \Q)'$ for the subspace of maps $f : \Sq([1]) \to \Q$ so that the restrictions $f|[1]_v$ and $f|[1]_h$ select equivalences. But one readily verifies that $f|[1]_v$ is an equivalence if and only if $f|[1]_h$ is an equivalence. So the top and bottom right squares in the above diagram are pullback squares. The legs of the right span in the diagram are equivalences on account of \cite[Theorem 4.13]{CompJaco}. Here we used that $\Q$ is locally complete. From the above, it now follows that the legs of the middle span are equivalences as well. We conclude that all the left horizontal maps in the diagram are equivalences by 2-out-of-3.
\end{proof}

\begin{proof}[Proof of \ref{thm:uni prop squares hor}]
    Since $\Sq(\C)$ has all companions, we equivalently have to show that 
    $$
    \map_{\DbliCat}(\Sq(\C), \Q') \to \map_{\DbliCat}(\C_h, \Q')
    $$
    is a monomorphism with the prescribed image where $\Q' := \tau_\mathrm{hcomp}\Q \times_\Q \tau_\mathrm{vcomp}\Q$. The horizontal inclusion $\C_h \to \Sq(\C)$ is equivalent to the inclusion 
    $$
       (\C^\op)_v^{\tp, \hop,\vop} \to \Sq(\C^\op)^{\tp, \hop, \vop},
    $$
    as can be shown from \cite[Construction 3.57]{CompJaco}.
    On account of \ref{lem:comp core is complete}, $\Q'$ is a complete double $\infty$-ca\-te\-gory. Hence, we can apply dualities to identify the above map with 
    $$
    \map_{\DbliCat}(\Sq(\C^\op), (\Q')^{\tp, \hop, \vop}) \to \map_{\DbliCat}((\C^\op)_v, (\Q')^{\tp, \hop, \vop}).
    $$
    Thus the desired conclusion follows from \ref{cor:uni prop squares vert}.
\end{proof}

\section{The Gray tensor product via squares}\label{section:comparison}

As announced in the introduction, we will now settle the following description of the Gray tensor product that was conjectured by Gaitsgory--Rozenblyum in \cite[Subsection 10.4.5]{GR}:

\begin{theorem}\label{theo:comparaison of gray tensor}
    Let $\C$ and $\D$ be $(\infty,2)$-categories. Then there exists a unique natural equivalence 
    $$
    \Gr(\C_h \times \D_v) \simeq \C \otimes \D.
    $$
\end{theorem}

We note that the unicity will be automatic on account of \ref{prop:auto of gray}.
To prove the above result, we will employ several density arguments to reduce to easier shapes. As an easy case, we have the following intermediate result:

\begin{lemma}
\label{lemma:comparaions step 0.5}
The desired natural equivalence of \ref{theo:comparaison of gray tensor} exists if $\C$ and $\D$ are $\infty$-categories.
\end{lemma}
\begin{proof}
Since $\Delta$ is dense in $\infty\Cat$, and the bifunctors $\otimes$ and $\Gr((-)_h \times (-)_v)$ preserve colimits in each variable, we may reduce to the case that $\C = [n]$ and $\D = [m]$ for some non-negative integers $n,m$. In this case, it directly follows from \ref{cons: realization-nerve adjunction}.
\end{proof}

Before moving to the general case, the next step will be to prove \ref{theo:comparaison of gray tensor} in case that either $\C$ or $\D$ is an $\infty$-category. We recall the following facts:

\begin{lemma}
\label{lemma:easy diagram chasing}
Suppose that we are given a commutative diagram
\[\begin{tikzcd}
	a & c & x \\
	b & d & y
	\arrow["u", from=1-1, to=1-2]
	\arrow[from=1-1, to=2-1]
	\arrow[from=1-2, to=1-3]
	\arrow[from=1-2, to=2-2]
	\arrow[from=1-3, to=2-3]
	\arrow["v"', from=2-1, to=2-2]
	\arrow[from=2-2, to=2-3]
\end{tikzcd}\]
in an $\infty$-category with pushouts.
If the outer square is a pushout square, and if both $u$ and $v$ are epimorphisms, then the right square is a pushout square as well.
\end{lemma}

\begin{proof}
This is an easy diagram chase using the pasting law for pushout squares.
\end{proof}

\begin{lemma}
\label{lemma:important epis}
Let $a$ and $b$ be globular sums. Then the  canonical maps $a \otimes b \to a \times b$ and $a \times b \to b$ are both epimorphisms in $(\infty,2)\Cat$.
\end{lemma}
\begin{proof}
It follows from \cite[Lemma 1.6.12]{Effectivity} that the map $a \otimes b \to a \times b$ is obtained by localizing all the cells of the shape $x\otimes y$ where $x,y$ are cells of non-negative dimension of respectively $a$ and $b$. The desired result then follows from \cite[Proposition 2.2.1.50]{loubaton2024categorical}.

For the second map, it suffices to show that $a\to [0]$ is an epimorphism as the cartesian product preserves colimits. This map witnesses a groupoidification, so it follows again from \cite[Proposition 2.2.1.50]{loubaton2024categorical}. 
\end{proof}

\begin{construction}\label{cons:funny equation}
We will now construct natural pushout squares
\[
    \begin{tikzcd}
	\tau_0[n]\times ([m]\times [k]) & {[n]\otimes ([m]\times [k])} \\
	\tau_0[n] \times [k] & {[n;m]\otimes [k]},
	\arrow[""{name=0, anchor=center, inner sep=0}, from=1-1, to=1-2]
	\arrow[from=1-1, to=2-1]
	\arrow[from=1-2, to=2-2]
	\arrow[from=2-1, to=2-2]
\end{tikzcd} \quad 
\begin{tikzcd}
 {[k]}\times [m]^\op \times \tau_0[n] & {([k]\times [m]^\op)\otimes[n]} \\
 {[k]} \times \tau_0[n]& {[k]\otimes[n;m]}.
	\arrow[""{name=0, anchor=center, inner sep=0}, from=1-1, to=1-2]
	\arrow[from=1-1, to=2-1]
	\arrow[from=1-2, to=2-2]
	\arrow[from=2-1, to=2-2]
\end{tikzcd}
\]
in $(\infty,2)\Cat$. We will just construct the left one, the construction of the right one is formally dual.

The first step will be to construct the collapse map $[n] \otimes ([m]\times [k]) \to [n;m] \otimes [k]$. Let $\psi : [n]\otimes[m] \to [n;m]$ be the map that appears in the square of \ref{lemma:globular sum as quotient of Gray}. Then we demonstrate that the dashed factorization of $(\infty,2)$-categories displayed in the left-hand side triangle exists:
\[
\begin{tikzcd}
        {[n]} \otimes [m] \otimes [k] \arrow[r, "{\psi \otimes [k]}"]\arrow[d] & {[n;m]} \otimes [k] \\ 
        {[n]} \otimes ([m] \times [k]) \arrow[ur, dashed]
\end{tikzcd} \quad
\Leftrightarrow \quad 
\begin{tikzcd}
        {[n]} \otimes_G [m] \otimes_G [k] \arrow[r, "{\psi \otimes_G [k]}"]\arrow[d] & {[n;m]} \otimes_G [k]. \\ 
        {[n]} \otimes_G ([m] \times [k]) \arrow[ur, dashed]
\end{tikzcd}
\]
This factorization will be unique since the left vertical map is an epimorphism by \ref{lemma:important epis}.
By Maehara's result recalled as \ref{theo:Gray tensor of globular sums}, $[n;m] \otimes [k]$ is 2-gaunt. So, by the adjunction $\Str\dashv N$ of \ref{def:discrete} and the monoidality of $\Str$, a triangle on the left-hand side in $(\infty,2)\Cat$ corresponds precisely to a triangle on the right-hand side in $2\Gaunt$. Let $u$ be a $0$-cell in $[n]$ and $v$ a $2$-cell in $[m]\otimes_G[k]$. Tracing through the construction of $\psi$, one 
sees that $u\otimes v$ is carried to an identity by $\psi \otimes_G [k]$. So it follows from \ref{rem:when factor throug tau1} that the desired dashed factorization exists.

Similar reasoning shows that the produced factorization fits in a canonical commutative diagram 
\[ 
\begin{tikzcd}
	{\tau_0[n] \times [m]\otimes [k]}\arrow[d]\arrow[r] & {[n]\otimes[m]\otimes[k]} \arrow[d] \\
	{\tau_0[n]\times [m]\times [k]}\arrow[d]\arrow[r] & {[n]\otimes([m]\times[k])} \arrow[d] \\
	{\tau_0[n]\otimes [k]} \arrow[r] & {[n;m]\otimes [k]},
\end{tikzcd}
\]
and the bottom square is uniquely determined since the top left and top right maps are epimorphisms. 
\end{construction}

\begin{lemma}
\label{lemma:funny equation}
The squares constructed in \ref{cons:funny equation} are pushout squares.
\end{lemma}
\begin{proof}
    This follows directly from applying \ref{lemma:globular sum as quotient of Gray} and \ref{lemma:easy diagram chasing}. 
\end{proof}

We now have the following intermediate result:

\begin{lemma}
\label{lemma:comparaions step 1}
The desired natural equivalence of \ref{theo:comparaison of gray tensor} exists if $\C$ is an $\infty$-category and $\D$ is an $(\infty,2)$-category.
\end{lemma}
\begin{proof}
By density, we may reduce to the case that $\C = [k] \in \Delta$ and $\D = [n;m] \in [\Delta;\Delta]$. It follows from the construction of the vertical and horizontal inclusion
that $\C_h\times \D_v$ is naturally the colimit of the span
\[
\begin{tikzcd}
    {[k]_h} \times \tau_0[n]_v & \arrow[l]  ([k] \times [m]^\op)_h \times \tau_0[n]_v \arrow[r] & ([k] \times [m]^\op)_h \times [n]_v.
\end{tikzcd}
\]
On account of \ref{lemma:comparaions step 0.5}, the functor $\Gr$ carries this to the colimit of the span
\[
\begin{tikzcd}
     {[k]} \times \tau_0[n] & \arrow[l]  [k] \times [m]^\op \times \tau_0[n] \arrow[r] & ([k] \times [m]^\op)\otimes [n],
\end{tikzcd}
\]
and this is computed by $[k] \otimes [n;m]$ on account of \ref{lemma:funny equation}.
\end{proof}

To prove \ref{theo:comparaison of gray tensor}, we need some additional combinatorial input.

\begin{lemma}
\label{lemma:crush product}
Let $n,m,k$ be non negative integers.
Then both the right and outer square in the canonical commutative diagram
\[\begin{tikzcd}
	{\tau_1[n;m]\otimes [k]} & {\tau_1[n;m]\times [k]} & {\tau_1[n;m]} \\
	{[n;m] \otimes [k]} & {[n;m] \times [k]} & {[n;m]}
	\arrow[from=1-1, to=1-2]
	\arrow[from=1-1, to=2-1]
	\arrow[from=1-2, to=1-3]
	\arrow[from=1-2, to=2-2]
	\arrow[from=1-3, to=2-3]
	\arrow[from=2-1, to=2-2]
	\arrow[from=2-2, to=2-3]
\end{tikzcd}\]
are pushout squares in $(\infty,2)\Cat$.
\end{lemma}
\begin{proof}
Recall that \ref{lemma:important epis} implies that the horizontal maps in the left-hand  square are epimorphisms, so that it suffices that the outer square is a pushout square by \ref{lemma:easy diagram chasing}.

Using \ref{cons:funny equation}, we may construct a canonical commutative cube
\[
	\begin{tikzcd}[column sep = 0pt, row sep = small]
			[n] \otimes (\tau_0[m] \times [k]) \arrow[dr]\arrow[rr]\arrow[dd] && {[n]} \times \tau_0[m] \arrow[dr]\arrow[dd] \\
			& |[alias=f]|\tau_1[n;m] \otimes [k] \arrow[rr, crossing over] && \tau_1[n;m] \arrow[dd]\\ 
			{[n]} \otimes ([m] \times [k]) \arrow[dr]\arrow[rr] && {[n]}\otimes [m]\arrow[dr] \\
			& |[alias=t]| {[n;m]} \otimes [k] \arrow[rr] &&  {[n;m]},
			\arrow[from=f,to=t, crossing over]
	\end{tikzcd}
\]
so that the front face is precisely the square we are interested in. The bottom face is a pushout on account of \ref{lemma:funny equation}. The arrows in the top face that go from the back to the front face, are both epimorphisms. By \ref{lemma:easy diagram chasing}, it suffices to check that the back face is a pushout square. But this follows from the fact that the square 
\[
    \begin{tikzcd}
        \tau_0[m] \times [k] \arrow[r]\arrow[d] &  \tau_0[m] \arrow[d] \\
        {[m]} \times [k] \arrow[r] & {[m]},
    \end{tikzcd}
\]
is a pushout in $\infty\Cat$, as can be directly verified by reducing to $m=k=1$. 
\end{proof}

\begin{lemma}

\label{lemma:comparaions step 3}
Let $n,m,k,l$ be non-negative integers.
Then the canonical commutative square
\[\begin{tikzcd}
	{\tau_0[n]_h \times [m]_v \times\tau_1[k;l]_v} & {\langle n,m \rangle \times [k;l]_v} \\
	{\tau_0[n]_h\times\tau_1[k;l]_v} & {[n;m]_h\times[k;l]_v}
	\arrow[""{name=0, anchor=center, inner sep=0}, from=1-1, to=1-2]
	\arrow[""{name=0p, anchor=center, inner sep=0}, phantom, from=1-1, to=1-2, start anchor=center, end anchor=center]
	\arrow[from=1-1, to=2-1]
	\arrow[from=1-2, to=2-2]
	\arrow[from=2-1, to=2-2]
\end{tikzcd}\]
is a pushout in $\CDbliCat$.
\end{lemma}
\begin{proof}
We may expand the commutative square to the commutative diagram
\[\begin{tikzcd}
	{\tau_0[n]_h\times [m]_v \times\tau_1[k;l]_v} & {\tau_0[n]_h \times [m]_v \times[k;l]_v} & {\langle n,m \rangle \times [k;l]_v} \\
	{\tau_0[n]_h\times\tau_1[k;l]_v} & {\tau_0[n]_h\times[k;l]_v} & {[n;m]_h\times[k;l]_v,}
	\arrow[from=1-1, to=1-2]
	\arrow[from=1-1, to=2-1]
	\arrow[from=1-2, to=1-3]
	\arrow[from=1-2, to=2-2]
	\arrow[from=1-3, to=2-3]
	\arrow[from=2-1, to=2-2]
	\arrow[from=2-2, to=2-3]
\end{tikzcd}\]
of which the outer square is of our interest. 
Note that the right square is a pushout square. Thus we must demonstrate that the left square is a pushout square. This follows directly from \ref{lemma:crush product} and the fact that the vertical inclusion $(-)_v : (\infty,2)\Cat \to \CDbliCat$ preserves colimits.
\end{proof}

We have the following extension of \ref{lemma:funny equation}:

\begin{lemma}
\label{lemma:comparaions step 2}
Suppose that $\P$ is a double $\infty$-category, and $\C$ an $\infty$-category.
Then there exists a natural equivalence
$$\textstyle   \P_0 \bigcup_{\P_0\times \C} \Gr(\P \times \C_v) \simeq \P_0\bigcup_{\C^{\op}\otimes \P_0}\C^{\op}\otimes \Gr(\P)$$
of $(\infty,2)$-categories.
\end{lemma}
\begin{proof}
    Both sides are indexed by functors $\DbliCat \times \infty\Cat \to (\infty,2)\Cat$ that preserve colimits in each variable. Hence, by density, we may reduce to the case that $\P = \langle n,m \rangle$ and $\C = [k]$. In this case, the inclusion $\P_0 \to \Gr(\P)$ of the vertical $\infty$-category is given by the obvious map $\tau_0[n] \times [m] \to [n]\otimes[m]$. Since $\Gr$ preserves colimits, the left-hand side is computed by 
    $$\textstyle \tau_0[n]\times [m] \bigcup_{\tau_0[n]\times ([m]\times[k])}[n]\otimes([m]\times [k]).$$
    This pushout is computed by $[n;k]\otimes[m]$ by \ref{lemma:funny equation}.
    The right-hand part of the equation corresponds to the pushout 
    $$\textstyle \tau_0[n]\times [m] \bigcup_{[k]^{\op}\otimes \tau_0[n]\otimes[m]}[k]^{\op}\otimes [n]\otimes[m].$$
    As the Gray tensor product commutes with colimits in both variables, \ref{lemma:globular sum as quotient of Gray} implies that this pushout is also  naturally equivalent to $[n;k]\otimes[m]$.
\end{proof}

We are now ready to complete the proof of \ref{theo:comparaison of gray tensor}.

\begin{proof}[Proof of \ref{theo:comparaison of gray tensor}]
The unicity follows from \ref{prop:auto of gray}.
By density, we may reduce to the case where $\C$ is of shape $[n;m] \in [\Delta;\Delta]$ and $\D$ of shape $[k;l]\in [\Delta;\Delta]$. We set $\P:= [n]_h\times [k;l]_v$. Note that the inclusion $(\P_0)_v \to \P$ is equivalent to the inclusion $\tau_0[n]_h \times \tau_1 [k;l]_v \to [n]_h \times [k;l]_v$.
We may use \ref{lemma:comparaions step 3} and \ref{remark:Gr and completion} to write  $\Gr(\C_h \times \D_v)$ as the colimit of the natural span
\[\begin{tikzcd}
	{\P_0} & {\P_0 \times [m]} & {\Gr(\P \times [m]_v)}.
	\arrow[from=1-2, to=1-1]
	\arrow[from=1-2, to=1-3]
\end{tikzcd}\]

It follows from \ref{lemma:comparaions step 1} and \ref{lemma:comparaions step 2} that there is a natural equivalence between $\Gr(\C_h \times \D_v)$ and the colimit of the span
\[\begin{tikzcd}
	\tau_0[n] \times \tau_{1}[k;l] & {[m]}^{\op}\otimes \tau_0[n] \otimes \tau_1[k;l] & {[m]^{\op}\otimes[n]\otimes [k;l]}.
	\arrow[from=1-2, to=1-1]
	\arrow[from=1-2, to=1-3]
\end{tikzcd}\]
By \ref{lemma:crush product}, this colimit is equivalent to the colimit of the span
\[\begin{tikzcd}
	{\tau_0[n]\times [k;l]} & {[m]^{\op}\otimes\tau_0[n]\otimes [k;l]} & {[m]^{\op}\otimes[n]\otimes[k;l]}.
	\arrow[from=1-1, to=1-2]
	\arrow[from=1-2, to=1-3]
\end{tikzcd}\]
We may now apply \ref{lemma:globular sum as quotient of Gray} to conclude that $\Gr(\C_h \times \D_v)$ is computed by $\C \otimes \D = [n;m]\otimes [k;l]$. Since all steps in the comparison depended naturally on $\C$ and $\D$, the naturality of the comparison is also witnessed.
\end{proof}

\nocite{*}
\bibliographystyle{amsalpha}
\bibliography{gr}

\end{document}